\documentclass[12pt]{amsart}
\usepackage{amsmath, amsthm, amssymb}
\usepackage[headings]{fullpage}
\usepackage[pdfstartview={FitH}]{hyperref}
\usepackage{xypic}
\xyoption{all}
\usepackage{tikz}
\usetikzlibrary{arrows,fit,shapes.geometric,decorations.pathmorphing}

\usepackage{mathrsfs} 

\usepackage{color}   

 \numberwithin{equation}{section}
 
 \theoremstyle{plain}
 \newtheorem{theorem}[equation]{Theorem}
 \newtheorem{corollary}[equation]{Corollary}
 \newtheorem{lemma}[equation]{Lemma}
 
 \newtheorem{question}[equation]{Question}

 \theoremstyle{definition}
 \newtheorem{definition}[equation]{Definition}
 
 \newtheorem{example}[equation]{Example}
  
 \newtheorem{remark}[equation]{Remark}
 
 \newtheorem{hypothesis}[equation]{Hypothesis}

 \newcommand{\cat}[1]{\ensuremath{\mathsf{#1}}}
 \newcommand{\op}{\ensuremath{^\mathrm{op}}}

 \newcommand{\Set}{\cat{Set}}
 \newcommand{\Ring}{\cat{Ring}}
 \newcommand{\cRing}{\cat{cRing}}
 \newcommand{\Ringop}{\Ring\op}
 \newcommand{\cRingop}{\cRing\op}
 
 \newcommand{\Topos}{\cat{Topos}}
 \newcommand{\RingedTopos}{\cat{RingedTopos}}
 \newcommand{\RingedCat}{\cat{RingedCat}}
 \newcommand{\fpCat}{\cat{fpCat}}
 \newcommand{\Cat}{\cat{Cat}}
 \newcommand{\Sh}{\cat{Sh}}
 \newcommand{\ShR}{\cat{ShR}}
 \newcommand{\ShRing}{\cat{ShRing}}
 \newcommand{\RingObj}{\cat{RingObj}}
 
 \DeclareMathOperator{\Spec}{Spec}
 \DeclareMathOperator{\Sp}{Sp}

 \DeclareMathOperator{\Hom}{Hom}
 \DeclareMathOperator{\id}{id}
 
 \DeclareMathOperator{\Obj}{Obj}

 \newcommand{\C}{\mathcal{C}}
 \newcommand{\D}{\mathcal{D}}
 \newcommand{\X}{\mathcal{X}}
 \newcommand{\Y}{\mathcal{Y}}
 \newcommand{\ZZ}{\mathcal{Z}}

 \newcommand{\Z}{\mathbb{Z}}
 \newcommand{\complex}{\mathbb{C}}
 \newcommand{\M}{\mathbb{M}}

 \newcommand{\separate}{\bigskip}
 
 \renewcommand{\O}{\mathcal{O}}

\begin{document}

\title{Sheaves that fail to represent matrix rings}
\author{Manuel L. Reyes}
\address{Department of Mathematics\\
Bowdoin College\\
8600 College Station\\
Brunswick, ME 04011--8486}
\email{reyes@bowdoin.edu}
\urladdr{http://www.bowdoin.edu/~reyes/}

\dedicatory{Dedicated to Professor Tsit Yuen Lam on the occasion of his seventieth birthday.}

\date{November 22, 2013}

\begin{abstract}
There are two fundamental obstructions to representing noncommutative rings via sheaves.
First, there is no subcanonical coverage on the opposite of the category of rings that
includes all covering families in the big Zariski site.
Second, there is no contravariant functor $F$ from the category of rings to the category of  ringed
categories whose composite with the global sections functor is naturally isomorphic to the identity,
such that $F$ restricts to the Zariski spectrum functor $\Spec$ on the category of commutative
rings (in a compatible way with the natural isomorphism).
Both of these no-go results are proved by restricting attention to matrix rings.
\end{abstract}

\maketitle

\section{Introduction}
\label{sec:intro}

Recent results from~\cite{Reyes} and~\cite{BergHeunen:nogo, BergHeunen:extending}
show that there are fundamental obstructions to the extension of
the Zariski spectrum functor from commutative rings to noncommutative rings. 
Specifically, if $F$ is a contravariant functor from rings to sets (or topological spaces) whose restriction
to commutative rings is the usual prime spectrum $\Spec$, then $F(\M_n(\complex)) = \varnothing$ for
$n \geq 3$; see~\cite{Reyes}. 
There are various constructions that allow one to study ``spaces without points.'' For instance,
locales formalize the properties of the lattice of open sets on a topological space, and toposes formalize
the properties of the category of sheaves of sets on a topological space.
Thus one might wish to dodge this obstruction by considering the prime
spectrum as a locale or topos (i.e., as a ``pointless space''). But the same type of obstruction was shown to
hold for functors taking values in the categories of locales and toposes; see~\cite{BergHeunen:extending}.

(As the intended audience of this paper includes ring theorists, who are not necessarily
acquainted with some of the categorical terminology used here, later sections include
very basic and brief accounts of relevant topics with suggestions for further reading.)

Yet another attempt to avoid such obstructions would be to replace the spectrum of a commutative
ring $R$ with certain sheaves associated to it. Of course, there is more than one kind of sheaf associated
to a commutative ring $R$ and its spectrum. We focus on two specific types: 
\begin{itemize}
\item[(I)] The functor $\Hom_\cRing(-,R) \colon \cRingop \to \Set$ is a \emph{sheaf on the big Zariski site.}
\item[(II)] The spectrum of $R$ is equipped with its \emph{structure sheaf} $\O_{\Spec(R)}$, a particular
member of the \emph{category of sheaves of rings} on its underlying topological space.
\end{itemize}
By replacing $\Spec(R)$ with either of the objects in (I) or (II) above, one passes to a purely
categorical setting that avoids any need for ``points'' or ``open sets.''
One might hope that after the spectrum of a ring is recast in either of these perspectives, it would
easily extend to noncommutative rings.
But to the contrary, the main results of this paper provide obstructions to noncommutative extensions
of the spectrum in the vein of (I) and (II) above.

We now summarize the two major results to be proved below. The first obstruction concerns the
view of $\Spec(R)$ expressed in (I) above, and will be proved in Section~\ref{sec:coverage}. 
A \emph{coverage} on a category $\C$ is a designation of certain sets of morphisms
in $\C$ as ``covering families'' and allows one to view certain contravariant set-valued
functors on that category as \emph{sheaves}.
A coverage $J$ on a category $\C$ is called \emph{subcanonical} if, for all
$X \in \C$, the representable functor $\Hom_\C(-,X) \colon \C\op \to \Set$ is a sheaf for $J$. 

\begin{theorem}
\label{thm:no coverage}
There is no subcanonical coverage $J$ on $\Ringop$ such that
every covering family for the big Zariski site is a covering family for $(\Ringop, J)$.
\end{theorem}

Now we state the second obstruction, which is related to the view of $\Spec(R)$ given in~(II) above.
A \emph{ringed category} $(\X, \O)$ is category $\X$ with finite products, equipped with a
ring object $\O$ in $\X$. These form the objects of a category $\RingedCat$, which is
described in detail in Definition~\ref{def:ringed categories}.
The \emph{ring of global sections} of a ringed category $(\X,\O)$ is the ring $\Hom_\X(1_\X,\O_\X)$,
where $1_\X$ is the terminal object of $\X$.
Ringed categories form a broad generalization of ringed spaces~\cite[II.2]{Hartshorne}.

\begin{theorem}
\label{thm:ringed category}
Let $F \colon \Ringop \to \RingedCat$ be a contravariant functor from the category of rings
to the category of ringed categories, whose restriction to the category of commutative rings
is isomorphic to $\Spec$. Suppose that there exists a natural transformation of functors
\[
\eta \colon \id_{\Ring} \to \Gamma \circ F
\]
(where $\Gamma \colon \RingedCat\op \to \Ring$ is the global sections functor), which restricts
on $\cRing$ to the canonical isomorphism from a commutative ring to the ring of global sections
of its Zariski spectrum.
Then for any ring $k$ and integer $n \geq 2$, the ringed category $F(\M_n(k))$
has the zero structure ring object (and therefore has zero global sections).
In particular, $\eta$ cannot be a natural isomorphism.
\end{theorem}

This result is proved in Section~\ref{sec:sheaves of rings}. A related obstruction was proved
in~\cite{BergHeunen:extending}; see Remark~\ref{rem:ringed topos} for a more
detailed comparison of that result with the one above.

Theorems~\ref{thm:no coverage} and~\ref{thm:ringed category} actually follow from
more general (but also more technical) results stated in each section. 

\subsection*{Acknowledgments}

It is my pleasure to thank Benno van den Berg, Luisa Fiorot, Chris Heunen, and Matthew Satriano for
helpful discussions, suggestions, and comments.

\subsection*{Conventions}

In this paper, all rings and ring homomorphisms are unital.
The categories of sets, rings, and commutative rings are respectively denoted by
$\Set$, $\Ring$, and $\cRing$.
We view a contravariant functor from a category $\C$ to a category $\D$ equivalently
as an arrow-reversing functor $\C \to \D$ or as an arrow-preserving functor $\C\op \to \D$.
We view all terminal objects of a category $\C$, which are canonically isomorphic,
as being ``equal'' for the sake of simplicity.

\section{A diagram of rings}
\label{sec:diagram}

The major results of this paper reduce to a basic diagram in the category of rings.

We introduce some notation to be used in the remainder of this paper. Let $k$ be a ring
and $n \geq 1$ an integer. We let $k^n = \prod_{i=1}^n k$ denote the $n$-fold product.
For $0 \leq i \leq n$ let $\pi_i \colon k^n \to k$ denote the projection map onto
the $i$th coordinate.  
In the ring $k^n$, let $e_i$ denote element whose $i$th entry is $1$ and whose other
entries are zero. In the matrix ring $\M_n(k)$ we let $E_{ij}$ denote
the matrix whose $(i,j)$-entry is $1$ and whose other entries are zero.
Also, we let $d \colon k^n\to\M_n(k)$ denote the usual diagonal embedding: the
$k$-linear map sending $e_i\mapsto E_{ii}$. 

\begin{lemma}
\label{lem:emptypreimage}
Let $k$ be a ring, and let $n\geq 2$ and $1 \leq i \leq n$ be integers. With notation as above,
the diagram
\[
\xymatrix{
k^n \ar[r]^-{\pi_i} \ar[d]_-d & k \ar[d] \\
\M_n(k) \ar[r] & 0
}
\]
is a pushout in the category of rings.
\end{lemma}

\begin{proof}
Fix an integer $\ell \neq i$ with $1 \leq \ell \leq n$.
Suppose that $R$ is a ring with homomorphisms $f \colon \M_n(k) \to R$ and
$g \colon k \to R$ such that $f d = g \pi_i$. In the ring $\M_n(k)$, we have
$1 = \sum_{j=1}^n E_{j\ell}E_{\ell \ell}E_{\ell j}$, so that $d(e_\ell)=E_{\ell \ell}$
generates the unit ideal.
Applying $f$, it follows that $fd(e_\ell)$ generates the unit ideal in $R$. 
On the other hand $\pi_i(e_\ell)=0$, so that $g\pi_i(e_\ell)=0$.
Now the image of $e_\ell$ in $R$ both generates the unit ideal and is zero, so
$R = 0$ and the diagram above is a pushout.
\end{proof}

We will be concerned with the pullback diagram in $\Ringop$ that is opposite to the diagram
in Lemma~\ref{lem:emptypreimage}.
Let us speak intuitively about this ``picture'' in the case where $k = \complex$ and $n = 2$.
For the moment, we will write $\Spec(\M_2(\complex))$ for an imaginary geometric object
corresponding to the matrix ring $\M_2(\complex)$. The pullback diagram is the following:
\[
\begin{tikzpicture}[anchor=base,auto]
  \path 
    (0,0) node (Spec0) {$\Spec(0)$} +(4,0) node (Speck) {$\Spec(\complex)$}
        +(0,-2) node (SpecM2) {$\Spec(\M_2(\complex))$}
        +(4,-2) node (Speck2) {$\Spec(\complex) \coprod \Spec(\complex)$} +(5,-1) node (AR1) {}
    ++(7.5,0) node (empty) {$\varnothing$} +(2,0) node (point) {$\{*\}$} +(0,-2) node (weird) {$??$}
		+(2,-2) node (twopoints) {$\{*, \bullet\}$} +(-1,-1) node (AR2) {}
  ;
  \draw[->] (Spec0) edge (Speck) ;
  \draw[->] (Spec0) edge (SpecM2) ;
  \draw[->] (Speck) edge (Speck2) ;
  \draw[->] (SpecM2) edge (Speck2) ;
  
  \draw[->] (empty) edge (point) ;
  \draw[->] (empty) edge (weird) ;
  \draw[->] (point) edge (twopoints) ;
  \draw[->] (weird) edge (twopoints) ;

  \draw[decorate,decoration={snake},->] (AR1) -- (AR2);
\end{tikzpicture}
\]
The maps $\Spec(\complex) \to \Spec(\complex^2) = \Spec(\complex) \coprod \Spec(\complex)$
send the single point to one of two points. Traditionally, the pullback diagram above gives the
fiber of the ``space'' $\Spec(\M_2(\complex))$ over either point.
Lemma~\ref{lem:emptypreimage} implies that both of these fibers are empty.  In other words, we
can imagine that $\Spec(\M_2(\complex))$ maps to the two-point space $\Spec(\complex^2)$,
``without hitting either point.''

Of course, the intuitive discussion above does not constitute rigorous mathematics. But the
remainder of this paper can be viewed as providing two different ways to turn these ideas into
precise results that obstruct certain approaches to realizing the category $\Ringop$ as a category
of ``spaces.''

The following remark, related to the triviality of the pushout of diagram above,
was kindly communicated to us by Luisa Fiorot: fibered coproducts in the category
$\cRing$ (which are given by tensor products) commute with finite products, but
Lemma~\ref{lem:emptypreimage} shows that fibered coproducts in the category
$\Ring$ (``amalgamated free products'') do not similarly commute with products. A
noncommutative algebraist may see in this the difference between idempotents in
commutative versus noncommutative rings: an idempotent in a commutative ring
``neatly splits'' the ring into a direct product of two commutative rings, while a
non-central idempotent in a noncommutative ring leads only to a Pierce corner
decomposition which is not a direct product of rings.

\section{The first obstruction: sheaves on a site}
\label{sec:coverage}

We refer the reader to~\cite[C.2.1]{Johnstone:elephant2} for a thorough account of
the theory of coverages and sheaves. 

Let $\C$ be a category with pullbacks. A \emph{coverage} $J$ on $\C$ is a rule assigning
to every object $X\in\C$ a collection $J(X)$ of sets of morphisms in $\C$ with codomain
$X$ (called ``covering families'') subject to the property:
\begin{itemize}
\item If $\{f_i \colon U_i \to X\} \in J(X)$ is a covering family and $g \colon Y \to X$
is a morphism in $\C$, then the family of pullbacks $\{g \times_X f_i \colon Y \times_X U_i \to Y\}$
is a covering family. 
\end{itemize}
A \emph{site} is a pair $(\C, J)$ where $\C$ is a category and $J$ is a coverage on $\C$.

(In the literature, a coverage is often defined by a weaker condition that can be
stated even if $\C$ does not have pullbacks~\cite[Definition~C.2.1.1]{Johnstone:elephant2}.
However, for categories with pullbacks such as $\cRingop$ or $\Ringop$, the sheaves on a
site $(\C, J)$ remain unchanged if $J$ is enlarged to be ``stable under pullback'' in the sense
above; this follows from~\cite[Lemma~C.2.1.6(i)]{Johnstone:elephant2}.)

\begin{example}
The relevant example for us is the \emph{big Zariski site} (see~\cite[VIII.6]{MacLaneMoerdijk}
and~\cite[Ex.~2.30]{Vistoli}). This is the site $(\cRingop,J)$
where for a commutative ring $R \in \cRingop$, the family $J(R)$ consists (up to isomorphism)
of all families opposite to those of the form
\[
\{R \to R[r_i^{-1}] \mid r_1, \dots, r_n \in R \mbox{ and } \sum r_i R = R\}.
\]
(Stated geometrically, these are the collections of open immersions onto distinguished open
subschemes of $\Spec(R)$ which collectively cover the space.) These families are stable
under pullback in $\cRingop$ because, for any homomorphism $g \colon R \to S$ in $\cRing$,
the corresponding pushout diagram in $\cRing$ is
\[
\xymatrix{
S[g(r_i)^{-1}] & R[r_i^{-1}] \ar[l] \\
S \ar[u] & R \ar[l]_g \ar[u]
}
\]
where $\sum r_i R = R$ implies that $\sum g(r_i) S = S$. We will refer to this coverage
as the \emph{Zariski coverage}.
\end{example}

For another example, a very special form of coverage used in algebraic geometry is a
\emph{Grothendieck topology}; see~\cite[Ch.~III]{MacLaneMoerdijk} and~\cite{Vistoli}.
While we will not recall the definition, a Grothendieck topology is a coverage that is closed
under certain saturation conditions. Each coverage ``generates'' a Grothendieck
topology that has the same sheaves as the original coverage;
see~\cite[Proposition~C.2.1.9]{Johnstone:elephant2}.

The primary purpose of a coverage is to define sheaves on the corresponding site. 
A \emph{sheaf} on a site $(\C,J)$ is a functor $F \colon \C\op \to \Set$ such that,
for every $X \in \C$ and every covering family $\{U_i \to X\} \in J(X)$, the
following diagram of sets is an equalizer for all $i$ and $j$:
\[
F(X) \to \prod_i F(U_i) \mathop{\rightrightarrows}^{F(p_1)}_{F(p_2)} F(U_i \times_X U_j),
\]
where $p_1 \colon U_i \times_X U_j \to U_i$ and $p_2 \colon U_i \times_X U_j \to U_j$
are the two ``projections'' from the pullbacks.

Recall that a coverage $J$ on a category $\C$ is called \emph{subcanonical} if, for all
$X \in \C$, the representable functor $\Hom_\C(-,X) \colon \C\op \to \Set$ is a sheaf for $J$. 
The important example for us is that \emph{the Zariski coverage on $\cRingop$ is 
subcanonical}~\cite[2.3.6]{Vistoli}.

\separate

We arrive at the first major result. It states that any coverage on $\Ringop$ which mildly
attempts to extend the Zariski coverage on $\cRingop$ will fail to distinguish matrix rings
from the zero ring.

We denote the pushout (or fibered coproduct) of two ring homomorphisms $R \to S$
and $R \to T$ by $S \ast_R T$.  This is a sort of ``amalgamated free product'' of $S$
and $T$ relative to $R$. If $k$ is a ring, then a \emph{ring over $k$} is a ring $R$
equipped with a homomorphism $k \to R$.

\begin{theorem}
\label{thm:confused sheaves}
Let $k$ be a ring and $n\geq 2$ be an integer. Let $J$ be a coverage on
the category $\Ringop$ for which the opposite of the family $\{\pi_i\colon k^n\to k \mid i=1,\dots,n\}$
is a covering family. 
Then for every ring $R$ over $k$, the opposite of the singleton family
$\{\M_n(R) \to 0\}$ is a covering family for $J$, and for any sheaf $F$ on the site
$(\Ringop,J)$, this morphism induces an isomorphism of sets $F(\M_n(R)) \cong F(0)$.
\end{theorem}

\begin{proof}
By the definition of a coverage, the pullbacks in $\Ringop$ (i.e., the pushouts in $\Ring$)
of the maps $\pi_i$ along the diagonal map $k^n \to\M_n(k)$ must form a covering family
for $J$. But by Lemma~\ref{lem:emptypreimage}, these pushouts are all zero maps
$\M_n(k)\to 0$. Thus the opposite family of $\{\M_n(k) \to 0\}$ is a covering family.
The homomorphism $k \to R$ induces a homomorphism $g \colon \M_n(k) \to \M_n(R)$,
and of course the pushout of the morphism $\M_n(k) \to 0$ along $g$ is 
$\M_n(R) \to 0$. As this is a pullback in $\Ringop$, the coverage axiom implies
that $\{\M_n(R) \to 0\}$ is a covering family for $J$.

Let $F$ be a sheaf on $(\Ringop, J)$; in particular, $F$ is a functor $\Ring = (\Ringop)\op \to \Set$.
The fibered coproduct $0 \ast_S 0$ in $\Ring$ for any ring $S$ (relative to the unique
homomorphism $S \to 0$) is easily seen to be zero.
Now the sheaf axiom declares that the following diagram must be an equalizer:
\[
F(\M_n(R)) \to F(0) \rightrightarrows F(0 \ast_{\M_n(R)} 0) = F(0).
\]
The two arrows $F(0) \to F(0)$ are equal because they are both the image under $F$ of the
unique ring homomorphism $0 \to 0$.
Thus the equalizer of these arrows is $F(0)$ itself, meaning that $F(\M_n(R)) \cong F(0)$.
\end{proof}

Suppose that $(\C, J)$ and $(\D, K)$ are sites and that $F \colon C \to D$ is a functor.
We say that $F$ \emph{preserves covering families} if, for every covering family
$\{f_i \colon U_i \to X\}$ for $(\C,J)$, 
the set $\{F(f_i) \colon F(U_i) \to F(X)\}$ is a covering family for $(\D,K)$ .

\begin{corollary}
\label{cor:no coverage}
There is no coverage $J$ on $\Ringop$ such that the functor
\[
F = \Hom_{\Ringop}(-,0) = \Hom_{\Ring}(0,-)\colon\Ringop\to\Set
\]
is a sheaf on $(\Ringop, J)$ and the inclusion $\cRingop \hookrightarrow \Ringop$
preserves Zariski covering families.
\end{corollary}

\begin{proof}
Assume for contradiction that such a subcanonical coverage $J$ exists. 
For any field $k$, the projections $\pi_i\colon k\times k\to k$ ($i=1,2$) together form a covering
family on the big Zariski site; on the level of schemes, these morphisms correspond to the open immersions
$\Spec(k)\to\Spec(k\times k)=\Spec(k)\coprod\Spec(k)$ mapping the unique point of $\Spec(k)$
onto either of the two points of $\Spec(k\times k)$. 
By hypothesis, this family must also cover $k\times k$ in the topology $J$ on $\Ringop$.
So by Theorem~\ref{thm:confused sheaves}, the sheaf $F$ assigns isomorphic sets
to $0$ and $\M_2(k)$. But this contradicts the fact that $\Hom_{\Ring}(0,0)$ is a singleton and
$\Hom_{\Ring}(0,\M_2(k))$ is empty.
\end{proof}

(One could prove slightly stronger statements of Theorem~\ref{thm:confused sheaves}
and Corollary~\ref{cor:no coverage} which assume only that the coverage $J$ on $\Ringop$
has a covering family which ``refines'' the family $\{\pi_i \colon k^n \to k\}$, or respectively
that $J$ refines all Zariski covering families. This would make use
of~\cite[Lemma~C.2.1.6(i)]{Johnstone:elephant2}.
We have fixed our definitions and stated our results in order to keep the exposition as
self-contained as possible.)

\separate

We obtain Theorem~\ref{thm:no coverage} as a special case of the preceding result.

\begin{proof}[Proof of Theorem~\ref{thm:no coverage}]
If $J$ is a subcanonical coverage on $\Ringop$, then the functor 
\[
\Hom_{\Ringop}(-,0)\colon\Ringop\to\Set
\]
is a sheaf for $(\Ringop, J)$. The claim immediately follows from Corollary~\ref{cor:no coverage}.
\end{proof}

The results proved above do \emph{not} indicate that every attempt to ``do geometry'' with the category
$\Ringop$ is futile. To the contrary, such approaches to noncommutative geometry have been
developed in~\cite{KontsevichRosenberg} and~\cite{Orlov}.
These approaches all happen to work with analogues (not generalizations) of phenomena from commutative
algebraic geometry, and deal with \emph{pre}sheaves of rings on the category $\Ringop$.
The results presented here simply show that such indirect approaches are necessary.

As each scheme $X$ defines a sheaf $\Hom(\Spec(-),X) \colon (\cRingop)\op \to \Set$ on the
big Zariski site, we see that the obstruction in Theorem~\ref{thm:no coverage} expresses, at least in
part, the difficulty in producing a notion of ``gluing'' of noncommutative spaces.

To close this section, we share an observation communicated to us by Benno van den Berg. 
It is possible to define sheaves even in the case when covering families are not stable under pullback
(and therefore do not form a coverage in the traditional sense).
For example, see the discussion surrounding~\cite[Example~C.2.1.13]{Johnstone:elephant2}.
The resulting categories of sheaves are generally much less well-behaved.
We do not know whether the obstructions proved in this section persist or can be avoided if one
works with such a generalized notion of sheaf on a category.

\section{The second obstruction: sheaves of rings}
\label{sec:sheaves of rings}

To present the final no-go result, let us imagine the most utopian setting for noncommutative
geometry.
There should be a category of noncommutative topological spaces.  (The need for such a category,
extending the usual category of ``commutative'' topological spaces, is suggested by the no-go results
presented in~\cite{Reyes} and~\cite{BergHeunen:extending}.)
Given a noncommutative topological space $X$, one would wish to have some collection of sheaves
on that space. Even in the event that there is some obstruction to general categories of ``sheaves
of sets on $X$'' in the fully noncommutative setting (see Question~\ref{q:obstruction}), one would
at least hope that there is a suitable category $\ShR(X)$ of ``sheaves of rings on $X$.''
Of course, there should be a global sections functor $\Gamma(X,-) \colon \ShR(X) \to \Ring$. 
Given a morphism of noncommutative spaces $f \colon X \to Y$
and a sheaf of rings $\O_X$ on $X$, one would wish for a ``direct image sheaf'' $f_* \O_Y$ that is a sheaf
of rings on $X$. 

In case $f \colon X \to Y$ is a continuous function between honest (``commutative'') topological
spaces, the direct image sheaf $f_* \O_X$ is defined by $f_* \O_X(U) = \O_X(f^{-1}(U))$ for each
open set $U \subseteq Y$. Thus, by the very definition of the direct image, $\Gamma(Y, f_* \O_X) = \Gamma(X, \O_X)$.
That is, \emph{direct images preserve global sections.} 

\begin{remark}
\label{rem:universe}
We will frequently refer to categories of ``large'' categories, whose morphisms are
functors and therefore might allow for ``large'' hom-sets. This is safely handled using
a ``second-order universe'' as described, for instance, in~\cite[\S 6.4]{Bergman}.
To be explicit, assume a Grothendieck universe $U_1$ and let $\Set$ be the category of
sets that are elements of $U_1$. Assume there is another universe $U_2$ such that
$U_1 \in U_2$.
Define $\Cat$ to be the category of all categories $\C$ such that $\Obj(\C) \in U_2$,
with functors for morphisms. (In particular, the category $\Set$ of ``small'' sets will be
an object of $\Cat$.) Then every ``category of categories'' we consider will be
a subcategory of $\Cat$.
These technicalities do not pose a serious issue, and we will mostly ignore size 
considerations in the rest of this paper.
\end{remark}

The following axioms express the minimal requirements for a category $\C$ to behave as a
category of ``ringed noncommutative spaces'' for which morphisms have ``direct images
of sheaves of rings that preserve global sections.''

\begin{hypothesis}
\label{hyp:category}
In this section, we will consider categories $\C$ satisfying the following properties:
\begin{itemize}
\item[(A)] Every object $X \in \C$ is a $3$-tuple $X = (\ShR(X), \O_X, \Gamma(X,-))$ such that:
  \begin{itemize}
  \item[(A1)] $\ShR(X)$ is a category;
  \item[(A3)] $\O_X$ is an object of $\ShR(X)$;
  \item[(A2)] $\Gamma(X,-) \colon \ShR(X) \to \Ring$ is a functor.
  \end{itemize}
\item[(B)] Every morphism $f \colon X \to Y$ in $\C$ is an ordered pair $f = (f_*, \underline{f})$ such that:
  \begin{itemize}
  \item[(B1)] $f_* \colon \ShR(X) \to \ShR(Y)$ is a functor satisfying 
    $\Gamma(Y,f_*(\O_X)) \cong \Gamma(X,\O_X)$;
  \item[(B2)] $\underline{f} \colon \O_Y \to f_* \O_X$ is a morphism in $\ShR(Y)$.
  \end{itemize}
\end{itemize}
\end{hypothesis}

The prototypical example of a triple satisfying axiom (A) is $(\ShRing(X), \O_X, \Gamma(X,-))$
where $X$ is a topological space, $\ShRing(X)$ is the category of sheaves of rings on $X$,
$\O_X \in \ShRing(X)$ is a particular sheaf of rings, and $\Gamma(X,-) \colon \ShRing(X) \to \Ring$ is
the global sections functor. We will return to this example (in much broader generality) when
we present the proof of Theorem~\ref{thm:ringed category} later in this section.

In the following, we will view the composite of two functors $F \colon \C_1\op \to \C_2$
and $G \colon \C_2\op \to \C_3$ as a \emph{covariant} functor $G \circ F \colon \C_1 \to \C_3$.

\begin{hypothesis}
\label{hyp:functor}
Let $\C$ be a category satisfying Hypothesis~\ref{hyp:category}. We will consider functors
$\Sp \colon \Ringop \to \C$ with the following properties:
\begin{itemize}
\item[(A)] The restriction of $\Sp$ to $\cRingop$ is the functor
\[
R \mapsto (\ShRing(\Spec(R)), \O_{\Spec(R)}, \Gamma(\Spec(R),-)).
\]
\item[(B)] There is a natural transformation $\eta \colon \id_\Ring \to \Gamma(\Sp(-),\O_{\Sp(-)})$
of endofunctors of $\Ring$ whose restriction to $\cRing$ is the isomorphism of $\id_\cRing$ with
$\Gamma(\Spec(-), \O_{\Spec(-)})$.
(That is, $\eta$ restricts to the unit of the adjunction between the $\Spec$ and global
sections functors).
\end{itemize}
\end{hypothesis}

For convenience, the following lemma is separated out from the proof of
Theorem~\ref{thm:zerosections}.
We continue to use the notation introduced in Section~\ref{sec:diagram}. 

\begin{lemma}
\label{lem:zerosheaf}
Let $k$ be a commutative ring, and let $\O$ be a sheaf of rings on the topological space
$\Spec(k^n)$ with a morphism of sheaves of rings $\O_{\Spec k^n} \to \O$
such that the induced map on global sections factors through the diagonal morphism
$d \colon k^n \hookrightarrow \M_n(k)$, as follows:
\[
\Gamma(\Spec(k^n), \O_{\Spec(k^n)}) \cong k^n \overset{d}{\longrightarrow} \M_n(k)
\to \Gamma(\Spec(k^n), \O).
\]
Then $\O$ is the zero sheaf of rings.
\end{lemma}

\begin{proof}
Let $U_j$ be the clopen subset of $X = \Spec k^n$ whose embedding $U_j \hookrightarrow X$ corresponds
to the projection $k^n \to k$ onto the $j$th factor, so that $X = \coprod_{j=1}^n U_j$.
Then for each $j$ the diagram
\[
\xymatrix{
k^n \ar@{=}[r] \ar[d]_{\pi_j} & \O_{\Spec k^n}(X) \ar@{^{(}->}[r]^-d \ar[d] & \M_n(k) \ar[r] & \O(X) \ar[d] \\
k \ar@{=}[r] & \O_{\Spec k^n}(U_j) \ar[rr] & & \O(U_j)
}
\]
is commutative.
It follows from Lemma~\ref{lem:emptypreimage} that $\O(U_j) = 0$. Because the open sets $U_j$ cover
$X$, it follows from the gluing axiom for sheaves that $\O(X) = 0$.
The fact that $\O$ is a sheaf of rings means that its restriction maps are ring
homomorphisms. So for all open subsets $U$ of $\Spec k^n$, the existence of a ring homomorphism
$0 = \O(X) \to \O(U)$ via restriction to $U$ implies that $\O(U) = 0$. Hence $\O$ is the zero sheaf.
\end{proof}

\begin{theorem}
\label{thm:zerosections}
Let $\C$ be a category satisfying Hypothesis~\ref{hyp:category} with a functor
$\Sp \colon \Ringop \to \C$ satisfying Hypothesis~\ref{hyp:functor}. 
Let $k$ be any ring and $n \geq 2$ an integer. Then
\[
\Gamma(\Sp(\M_n(k)), \O_{\Sp(\M_n(k))}) = 0.
\]
In particular, if $k$ is nonzero then $\Gamma(\Sp(\M_n(k)), \O_{\Sp(\M_n(k))}) \ncong \M_n(k)$.
\end{theorem}

\begin{proof}
It suffices to consider the case $k = \Z$. For, given any ring $k$, the unique homomorphism
$\Z \to k$ induces a homomorphism
\[
\Gamma(\Sp(\M_n(\Z)), \O_{\Sp(\M_n(\Z))}) \to \Gamma(\Sp(\M_n(k)), \O_{\Sp(\M_n(k))}).
\]
If the ring on the left is zero, then the ring on the right must also be zero.

Recall the notation for the diagonal morphism $d \colon \Z^n \hookrightarrow \M_n(\Z)$.
In the category $\C$, there is a resulting morphism $f = \Sp(d) \colon \Sp(\M_n(\Z)) \to \Sp(\Z^n)$.
This morphism consists of a functor $f_* \colon \ShR(\Sp(\M_n(\Z))) \to \ShR(\Sp(\Z^n)) = \ShRing(\Spec(\Z^n))$
along with a morphism $\underline{f} \colon \O_{\Sp(\Z^n)} = \O_{\Spec(\Z^n)} \to f_* \O_{\Sp(\M_n(\Z))}$
in the category $\ShRing(\Spec(\Z^n))$.

Write $X = \Sp(\M_n(\Z))$. 
Property~(B) of Hypothesis~\ref{hyp:functor} provides the following commutative diagram:
\[
\xymatrix{
\Z^n\  \ar@{=}[d] \ar@{^{(}->}[rr]^-d & & \M_n(k) \ar[d] \\
\Gamma(\Spec(\Z^n),\O_{\Spec \Z^n}) \ar[r]
& \Gamma(\Spec(\Z^n), f_*\O_X) \ar@{=}[r] &
\Gamma(X, \O_X) 
}
\]
where the vertical arrows are the components of the natural transformation $\eta$, and the
isomorphism $\Z^n \cong \Gamma(\Spec(\Z^n), \O_{\Spec(\Z^n)})$ coincides with the one
provided by the duality between commutative rings and affine schemes.
It follows from Lemma~\ref{lem:zerosheaf} that $f_* \O_{\Sp(\M_n(\Z))}$ is the zero sheaf.
From condition~(B1) of Hypothesis~\ref{hyp:category} we conclude that
\[
\Gamma(\Sp(\M_n(\Z)), \O_{\Sp(\M_n(\Z))}) = \Gamma(\Spec(\Z^n), f_* \O_{\Sp(\M_n(\Z))}) = 0. \qedhere
\]
\end{proof}

The ``preservation of global sections'' in Hypothesis~\ref{hyp:category}(B1) plays a crucial role
in the proof above. We do not know what can be said if that hypothesis is omitted or weakened
(or similarly, if the analogous hypothesis on $\fpCat$ below is omitted). The latter part of
Question~\ref{q:obstruction} is related to this.

\separate

A rather general category satisfying Hypothesis~\ref{hyp:category} can be obtained from
\emph{ringed categories}, which we now proceed to define.
Let $\X$ be a category with finite products. In particular, $\X$ has a terminal object $1_\X$, which
is the product indexed by the empty set. 
We define the \emph{global sections} functor $\Gamma(\X,-) \colon \X \to \Set$ to be the
hom-functor $\Gamma(\X,-) = \Hom_\X(1_\X,-)$.
(This agrees with the usual global sections functor in case $\X $ is a topos, such as the category
of sheaves on a topological space or locale~\cite[p~135]{MacLaneMoerdijk}.)

A \emph{ring object} of $\X$ is an object $\O \in \X$ that is equipped with ``zero,'' ``unity,''
``addition,'' ``subtraction,'' and ``multiplication'' morphisms
\[
0,1 \in \Hom_\X(1_\X ,\O), \quad +,\bullet \in \Hom_\X(\O \times \O , \O), \quad 
 \mbox{and} \quad - \in \Hom(\O,\O),
\]
such that the usual axioms for (associative, unital) rings hold when expressed as commuting
diagrams in $\X$ between products of copies of $\O$ (see~\cite[VIII..5]{MacLaneMoerdijk}).
A morphism of ring objects $\O_1$ and $\O_2$ in $\X$ is a morphism $\O_1 \to \O_2$ in
$\X$ such that the induced diagrams relating the zero, unity, addition, subtraction,
and multiplication of $\O_1$ and $\O_2$ are commutative. 

If $\X$ is a category with finite products, we write $\RingObj(\X)$ for the category of ring
objects in $\X$.
The structure of a ring object $\O \in \RingObj(\X)$ naturally induces a ring structure on
the global sections $\Hom_\X(A,\O)$ for any $A \in \X$ (essentially due to the fact that
the functor $\Hom_\X(A,-) \colon \X \to \Set$ preserves products).
In this way, the global sections functor $\Gamma(\X,-) = \Hom_\X(1_\X,-) \colon \X \to \Set$
induces a functor $\RingObj(\X) \to \Ring$, which we also denote by $\Gamma(\X,-)$.

If $f_* \colon \X \to \Y$ is a functor between categories with finite products that
preserves finite products, then $f_*$ induces a natural transformation
$\Gamma(\X,-) \to \Gamma(\Y, f_*(-))$ of functors $\X \to \Set$, as follows.
Because a terminal object is a product indexed by the empty set, we have
$f^*(1_\X) = 1_\Y$. Thus for any $A \in \X$, $f_*$ induces the arrow below:
\[
\Gamma(\X,A) = \Hom_\X(1_\X, A) \to \Hom_\Y(f_*(1_\X), f_*(A)) = \Hom_\Y(1_\Y, f_*(A)) = \Gamma(\Y,f_*(A)).
\]
We say that $f_*$ \emph{preserves global sections} if the transformation
above is a natural isomorphism $\Gamma(\X,-) \cong \Gamma(\Y,f_*(-))$.
We let $\fpCat$ denote the category whose objects are categories with finite
products and whose morphisms are functors that preserve finite products
and global sections.
(Recall Remark~\ref{rem:universe} regarding size issues.)

\begin{definition}
\label{def:ringed categories}
A \emph{ringed category} is a pair $(\X, \O)$ where $\X$ is a category with finite products
and $\O$ is a ring object in $\X$; we refer to $\O$ as the \emph{structure ring object} of
$(\X,\O)$. A \emph{morphism} of ringed categories $(\X,\O_\X) \to (\Y, \O_\Y)$ is a pair
$(f_*, \underline{f})$ where $f_* \colon \X \to \Y$ is a functor in $\fpCat$
and $\underline{f} \colon \O_\Y \to f_* \O_\X$ is a morphism in $\RingObj(\Y)$.
Given morphisms $(f_*, \underline{f}) \colon (\X,\O_\X) \to (\Y, \O_\Y)$ and
$(g_*, \underline{g}) \colon (\Y, \O_Y) \to (\ZZ, \O_\ZZ)$, we define their \emph{composite}
to be the morphism 
\[
(g_*, \underline{g}) \circ (f_*, \underline{f}) = (g_* \circ f_*, \underline{g} \circ g_*(\underline{f})).
\]
This operation is readily seen to be associative.
The category $\RingedCat$ has ringed categories for objects, with morphisms and
composition rule as above.
The assignment $(\X, \O_\X) \mapsto \Gamma(\X, \O_\X)$ defines a functor
$\Gamma \colon \RingedCat\op \to \Ring$, which we also refer to as the \emph{global
sections functor}.
\end{definition}

The \emph{zero ring object} of a category with finite products is the terminal object,
equipped with its unique structure as a ring object. The following observation will
prove useful.

\begin{lemma}
\label{lem:zero sections}
Let $(\X, \O)$ be a ringed category. If $\Gamma(\X,\O)$ is the zero ring, then $\O$
is the zero ring object on $\X$.
\end{lemma}

\begin{proof}
Because $\O$ is a ring object in $\X$, the Hom-functor
$\Hom_\X(-,\O)$ has the extra structure of a functor $\X\op \to \Ring$. Given
any object $A \in \X$, the morphism to the terminal object $A \to 1_\X$ induces
a ring homomorphism 
\[
0 = \Gamma(\X, \O_\X) = \Hom_\X(1_\X, \O_\X) \to \Hom_\X(A, \X).
\]
This implies that $\Hom_\X(A, \O)$ is the zero ring and consequently is a singleton
set for all $A$. Thus $\O$ is a terminal object, and therefore is a zero ring object.
\end{proof}

For the category $\Sh(X)$ of sheaves of sets on a topological space $X$, the category
of ring objects in $\Sh(X)$ is the same as the category of sheaves of rings on $X$;
see~\cite[II.7]{MacLaneMoerdijk}.
This provides a functor from the category of ringed spaces~\cite[II.2]{Hartshorne} to the
category of ringed categories, which is faithful.
In this way, we view the category of ringed spaces as a subcategory of $\RingedCat$.
By composing with this embedding, we may consider the Zariski spectrum as a functor
$\Spec \colon \cRingop \to \RingedCat$.

\separate

We are finally prepared to prove the second of the major results of Section~\ref{sec:intro}.

\begin{proof}[Proof of Theorem~\ref{thm:ringed category}]
Let $\C$ be the category defined as follows:
\begin{itemize}
\item The objects of $\C$ are ordered triples $(\RingObj(\X), \O_\X, \Gamma(\X,-))$,
where $(\X,\O_\X)$ is a ringed category and $\Gamma(\X,-) \colon \RingObj(\X) \to \Ring$
is the (co)restriction of the global sections functor on $\X$.
\item A morphism 
\[
f \colon (\RingObj(\X), \O_\X, \Gamma(\X,-)) \to (\RingObj(\Y), \O_\Y, \Gamma(\Y,-))
\]
is a morphism of ringed categories $f = (f_*, \underline{f}) \colon (\X,\O_\X) \to (\Y,\O_\Y)$.
\end{itemize}
The construction of $\RingedCat$ and the discussion preceding this proof indicate that
$\C$ satisfies Hypothesis~\ref{hyp:category}.
There is a rather obvious functor $G \colon \RingedCat \to \C$, defined by
$(\X,\O_\X) \mapsto (\RingObj(\X), \O_\X, \Gamma(\X,-))$ for objects and defined in the
trivial way on morphisms.

Let $F \colon \Ringop \to \RingedCat$ and $\eta \colon \id_\Ring \to \Gamma \circ F$ be
as in the statement of Theorem~\ref{thm:ringed category}, and define
$\Sp = G \circ F \colon \Ringop \to \C$. By construction of $G$, the functor
\[
\Gamma(\Sp(-), \O_{\Sp(-)}) \colon \Ring \to \Ring
\]
coincides with the composite $\Gamma \circ F \colon \Ring \to \Ring$. 
Thus $\eta$ is a natural transformation
$\id_\Ring \to \Gamma \circ F = \Gamma(\Sp(-), \O_{\Sp(-)})$.
The assumptions on $F$ and $\eta$ guarantee that the functor $\Sp$
and natural transformation $\eta$ satisfy Hypothesis~\ref{hyp:functor}.

Fix a ring $k$ and integer $n \geq 2$. Theorem~\ref{thm:zerosections}
implies that $\Gamma(\Sp(\M_n(k)), \O_{\Sp(\M_n(k))}) = 0$.
By the construction of $G$ and $\Sp = G \circ F$, one readily sees that this
equation means that the ringed category $F(\M_n(k)) = (\X, \O_\X)$ has ring of
global sections $\Gamma(\X, \O_\X) = 0$. It follows from Lemma~\ref{lem:zero sections}
that $\O_\X$ is the zero ring object of $\X$.
\end{proof}

\separate

A similar obstruction, for contravariant functors from the category of rings
to the category of ringed toposes, was presented by van den Berg and Heunen
in~\cite[Corollary~6.3]{BergHeunen:extending}. 
To compare Theorem~\ref{thm:ringed category} with their result, we briefly review
the necessary terminology.

A \emph{topos} can be defined tersely as a category with finite limits and power
objects. (This is sometimes called an \emph{elementary topos}, to distinguish from
the more specialized notion of a \emph{Grothendieck topos}.) The quintessential
example of a topos is the category of sheaves of sets on a fixed topological space.
We will not recall much of topos theory, but the reader is referred
to~\cite{MacLaneMoerdijk} for an in-depth treatment.

If $\X$ and $\Y$ are toposes, a \emph{geometric morphism} $f \colon \X \to \Y$ is a
pair $f = (f^*, f_*)$ such that $f^* \colon \Y \to \X$ is and $f_* \colon \X \to \Y$
are functors (respectively called the \emph{inverse image} and \emph{direct image}
of $f$) such that $(f^*,f_*)$ is an adjoint pair and $f^*$ preserves finite limits.
Let $\Topos$ denote the category whose objects are toposes and whose morphisms
are geometric morphisms.
Given a morphism $f \colon \X \to \Y$ in $\Topos$, the right adjoint functor
$f_* \colon \X \to \Y$ preserves limits, including finite products. We claim that
$f_*$ preserves global sections as well. Indeed, because $f^*$ preserves finite
limits and a terminal object is the limit of the empty diagram, we have
$f^*(1_\Y) = 1_\X$. So
\[
\Gamma(\X,A) = \Hom_\X(1_\X, A) = \Hom_\X(f^*(1_\Y), A) \cong \Hom_\Y(1_\Y, f_*(A)) = \Gamma(\Y,f_*(A)).
\]
This establishes the existence of a ``forgetful'' functor $\Topos \to \fpCat$ that
acts trivially on objects and sends a geometric morphism to its direct image part.

A \emph{ringed topos} is a pair $(\X, \O)$ where $\X$ is a topos and $\O$ is a ring object
in $\X$. A \emph{morphism} of ringed toposes $(\X,\O_\X) \to (\Y, \O_\Y)$ consists of a
pair $(f,\underline{f})$ where $f = (f^*, f_*) \colon \X \to \Y$ is a geometric morphism
and $\underline{f} \colon \O_\Y \to f_* \O_\X$ is a morphism of ring objects in $\Y$.
The category $\RingedTopos$ is the category of ringed toposes and their morphisms.
There is a ``forgetful'' functor $\RingedTopos \to \RingedCat$, which
acts trivially on objects and acts on morphisms by $(f, \underline{f}) \mapsto (f_*, \underline{f})$.
Clearly this functor preserves global sections  and reflects which ringed
toposes have trivial structure ring objects. It is now quite clear that 
\emph{Theorem~\ref{thm:ringed category} still holds when the category $\RingedCat$
is replaced by the category $\RingedTopos$.} 

\begin{remark}
\label{rem:ringed topos}
The result~\cite[Corollary~6.3]{BergHeunen:extending} of van den Berg and Heunen
states that any functor $F \colon \Ringop \to \RingedTopos$ that extends the usual Zariski
spectrum must assign the trivial object to $\M_n(\complex)$ for $n \geq 3$. Their conclusion
is much stronger than the claim that $\M_n(\complex)$ is assigned a ringed topos with
trivial structure sheaf, and actually stems from a \emph{topological} obstruction
about functors to the category of toposes~\cite[Corollary~6.2]{BergHeunen:extending}.
Their result also relies crucially on the Kochen-Specker Theorem~\cite{KochenSpecker}
(as did the results of~\cite{Reyes}).

By contrast, Theorem~\ref{thm:ringed category} (and its ringed topos version)
is of a more \emph{algebro-geometric} nature: while the topos (a topological structure)
assigned to $\M_n(k)$ may not be trivial, the corresponding sheaf of rings (which
enhances the topological structure to a geometric structure) must be zero.
Furthermore, because the theorem is valid for $n = 2$, it is clearly
independent of the Kochen-Specker Theorem.
\end{remark}

\separate

We conclude with a question that is suggested by Theorem~\ref{thm:ringed category}
and Remark~\ref{rem:ringed topos}. 
Let $\Cat$ denote the category of categories with functors for morphisms.
(Recall Remark~\ref{rem:universe} regarding size issues.)
We have already defined a forgetful functor $\Topos \to \fpCat$ sending
geometric morphisms to their direct image parts. 
Composing further with the forgetful functor $\fpCat \to \Cat$ yields
a forgetful functor $\Topos \to \Cat$. Taking inverse images of geometric 
morphisms ($f \mapsto f^*$) provides another forgetful functor $\Topos\op \to \Cat$.
By the \emph{trivial category} we mean the terminal object in
$\Cat$ (and $\fpCat$), the category with one object and one morphism.

We know that any functor $\Ringop \to \RingedTopos$ extending the Zariski
spectrum must assign the trivial object to $\M_3(\complex)$ by van den Berg
and Heunen's result.
Even if one tries to construct a functor $\Ringop \to \RingedCat$ extending
the spectrum but generalizing the underlying topos,
Theorem~\ref{thm:ringed category} shows that the ring object assigned
to $\M_3(\complex)$ must be trivial.
Since van den Berg and Heunen's obstruction stems from an underlying
topological (Kochen-Specker) obstruction, we ask whether the same kind of
underlying obstruction occurs in the case of $\RingedCat$. 

\begin{question}
\label{q:obstruction}
Fix an integer $n \geq 3$.
Let $F \colon \Ringop \to \fpCat$ be a functor whose restriction to
$\cRingop$ is the composite of $\Spec \colon \cRingop \to \Topos$
with the forgetful functor $\Topos \to \fpCat$. Does it follow that
$F(\M_n(\complex))$ is the trivial category?

More generally, suppose that $F \colon \Ringop \to \Cat$ (respectively,
$F \colon \Ring \to \Cat$) is a functor whose restriction to $\cRing$
is the composite of $\Spec \colon \cRingop \to \Topos$ with
the forgetful functor $\Topos \to \Cat$ (respectively, 
$\Topos\op \to \Cat$). What can be said about $F(\M_n(\complex))$?
\end{question}

\bibliographystyle{amsplain}
\bibliography{ZeroSheaves}
\end{document}